\documentclass[a4paper,11pt]{amsart}
\usepackage[english]{babel}
\usepackage[centertags]{amsmath}
\usepackage{latexsym, amsfonts, amssymb, graphicx}
\usepackage{picins}

\usepackage[all,dvips]{xy}
\usepackage{graphicx}
\usepackage[latin1]{inputenc}
\usepackage{url}
\usepackage{enumerate}
\usepackage{multirow}

\theoremstyle{plain}
\newtheorem{thm}{Theorem}[section]
\newtheorem{lem}[thm]{Lemma}
\newtheorem{prop}[thm]{Proposition}
\newtheorem{cor}[thm]{Corollary}

\theoremstyle{definition}
\newtheorem{defn}[thm]{Definition}

\theoremstyle{remark}
\newtheorem{rem}[thm]{Remark}

\def\Z{{\mathbf Z}}
\def\P{{\mathbf P}}

\def\L{{\mathcal L}}
\def\O{\mathcal{O}}

\newcommand{\F}{\mathbf{F}}

\newcommand{\Pic}{\textrm{Pic}}

\def\fract#1/#2{\hbox{\leavevmode
\kern.1em \raise .5ex \hbox{\the\scriptfont0 $#1$}\kern-.1em }/
\hbox{\kern-.15em \lower .25ex \hbox{\the\scriptfont0 $#2$}}
}

\begin{document}

\title[Rational surfaces yielding good codes]{Construction of rational surfaces yielding good codes}
\author{Alain \textsc{Couvreur}}

\address{INRIA Saclay \^Ile-de-France, Projet TANC -- École Polytechnique,
Laboratoire LIX, CNRS, UMR 7161,
91128 Palaiseau Cedex,
France}
\email{alain.couvreur@inria.fr}

\begin{abstract}
In the present article, we consider Algebraic Geometry codes on some rational surfaces.
The estimate of the minimum distance is translated into a point counting problem on plane curves. This problem is solved by applying the upper bound \emph{\`a la Weil} of Aubry and Perret together with the bound of Homma and Kim for plane curves.
The parameters of several codes from rational surfaces are computed. Among them, the codes defined by the evaluation of forms of degree $3$ on an elliptic quadric are studied. As far as we know, such codes have never been treated before. 
Two other rational surfaces are studied and
very good codes are found on them. In particular, a $[57,12,34]$ code over $\F_7$ and a $[91,18,53]$ code over $\F_9$ are discovered, these codes beat the best known codes up to now.
\end{abstract}

\maketitle

\thispagestyle{empty}



\section*{Introduction}
Algebraic Geometry codes have been first introduced by Goppa in \cite{goppa} in 1981.
A few time after, Tsfasman Vl\u{a}du\c{t} and Zink proved in \cite{TVZ} that some families of error--correcting codes beat the Gilbert--Varshamov bound.
This unexpected result motivated hundreds of publications on Algebraic Geometry codes.

Goppa's construction (\cite{goppa}) provides codes from algebraic curves. This approach is extended to arbitrary dimensional varieties by Manin in \cite{manin}.
However, only few results are known on codes on higher dimensional varieties.
Indeed, if the estimation of the minimum distance is an elementary task for codes on curves, it becomes a very hard problem in the higher dimensional case.
Therefore, most of the known works on codes from varieties of dimension at least $2$, deal with the estimate of the minimum distance of codes on varieties having some particular arithmetical or geometrical property.
Among the others (the list is not exhaustive), codes on quadric varieties are studied by Aubry in \cite{aubry}, the parameters of codes on Hermitian surfaces are computed in \cite{chak} and \cite{fred} and lower bounds for the minimum distance of codes from surfaces with a small arithmetical Picard number are computed by Zarzar in \cite{zarzar}.

The work of Zarzar \cite{zarzar} is of particular interest. It shows that surfaces with a small arithmetical Picard number (i.e. the Rank of the Neron--Severi group) provide in general codes with a good minimum distance for given length and dimension.
Basically, to have a high minimum distance, the global sections of the line bundle $\L$ used to produce the code should not vanish at too many rational points of the surface.
If the arithmetical Picard number is small, the vanishing locus of a global section of $\L$ cannot break into too many irreducible components and hence cannot have too many rational points.
The work of Zarzar should be compared with that of Aubry \cite{aubry} and Edoukou \cite{fred2} in which codes on elliptic quadrics turn out to be better than codes on hyperbolic quadrics. Recall that the first ones have arithmetical Picard number $1$ and the other ones arithmetical Picard number $2$.


Therefore, surfaces with a small arithmetical Picard number seem to be suitable to produce good codes. On the other hand, the estimate of the minimum distance remains a difficult task which is almost equivalent to a problem of estimating the maximal number of rational points of an element of a linear system of curves.

The purpose of the present article is to consider rational surfaces obtained by blowing up the projective plane at few closed points. Such a surface has a small Picard number. Moreover, since the surface is rational, the estimate of the minimum distance is translated into a problem of point counting for plane curves. For any curve, one can use the bound of Aubry and Perret \cite{AP2}. This bound is sharp when the base field is large. In addition, for plane curves and the bound from Homma and Kim \cite{HommaKim} is suitable and sharper that Aubry and Perret's one when the base field is small.

Using this approach, we first study codes on elliptic quadrics and are able to give a lower bound for the minimum distance of the codes obtained by evaluation of forms of degree $3$. As far as we know, this study has never been done up to now. Afterwards, we study the codes from two other rational surfaces. The first one (the surface $Y$) is the projective plane blown up at one rational point and a closed points of degree $4$. The second one (the surface $Z$) is obtained by blowing up the projective plane at one closed point of degree $3$.
Both surfaces provide good codes. In particular, the surface $Z$ yields a $[57,12,34]$ code over $\F_7$ and a $[91,18,53]$ code over $\F_9$ which both beat the best known codes given in \cite{grassl} and \cite{minT2}.




\subsection*{Outline of the article}
Prerequisites on Algebraic Geometry codes on surfaces and maximum number of rational points of a curve are recalled in Section \ref{S1}.
Codes on elliptic quadrics in $\P^3$ are studied in Section \ref{quad}, in particular, the parameters of the code obtained by the evaluation of forms of degree $3$ are estimated.
In Section \ref{const}, we present the construction of two other rational surfaces.
Explicit examples of codes on these surfaces are studied and turn out to be very good. In particular, the second surface (the surface $Z$) provides two codes which beat the best known codes up to now: a $[57,12,34]$ code over $\F_7$ and a $[91,18,53]$ code over $\F_9$. 

\section{Prerequisites}\label{S1}
In this section, we briefly recall some definitions and properties in algebraic geometry and algebraic geometric coding theory.
For further details, we refer the reader to \cite{H} and \cite{sch1} for algebraic geometry and to \cite{sti} and \cite{TVN} for Algebraic Geometry codes.

\subsection{Notations}In what follows, $X$ denotes a smooth projective geometrically irreducible surface over a finite field $\F_q$.

\subsubsection{Divisors, linear equivalence and intersection product}
The linear equivalence between two divisors $D,D'$ on $X$ is denoted by
$
D\sim D'.
$
The \emph{Picard group} of $X$, which is the group of linear equivalence classes of divisors, is denoted by $\textrm{Pic}_{\F_q}(X)$.
If $X$ is rational, then its Picard group is finitely generated and its rank is called the \emph{Picard number} of $X$.

One can define a natural pairing on $\textrm{Pic}_{\F_q}(X)$ called the \emph{intersection product} (\cite{H} Chapter V, Theorem 1.1). Given two divisor classes $D,D'$ on $X$, their intersection product is denoted by
$D.D'$. Moreover, we denote by $D^2$ the self-intersection of the class $D$, that is $D^2:=D.D$.

\subsubsection{Invertible sheaves and line bundles}
Recall that there is a one-to-one correspondence between linear equivalence classes of divisors, isomorphism classes of line bundles over $X$ and isomorphism classes of invertible sheaves on $X$ (\cite{sch2} Chapter VI \S 1.4).
Given a line bundle $\L$ over $X$, its space of global sections is denoted by $H^0 (X, \L)$.

Finally, given an integer $m$, we denote by $\O_X (m)$ the  \emph{$m$--th twisting sheaf} over $X$ (\cite{H} Chapter II, page 117).
If $m\geq 0$, then, given an embedding $X \hookrightarrow \P^r$, the space of global sections $H^0 (X, \O_X (m))$ is the space of the restrictions to $X$ of homogeneous polynomials of degree $m$ in $r+1$ variables.
To this sheaf corresponds a line bundle (up to isomorphism), which we also denote by $\O_X (m)$ for convenience's sake.

\subsection{Algebraic Geometry codes}
First, let us recall the definition of an Algebraic Geometry code on a surface.

\begin{defn}[Manin \cite{manin}]
Let $X$ be a  smooth projective geometrically irreducible surface over a finite field $\F_q$ and $\L$ be a line bundle over $X$.
Let $P_1, \ldots , P_n$ be the set of rational points of $X$.
  The code $C_L(X, \L)$ is defined as the image of the map
\begin{equation}\label{evalmap}
\textrm{ev}:\left\{
  \begin{array}{ccc}
    H^0 (X, \L) & \rightarrow & \bigoplus \L_{P_i} \simeq \F_q^n \\
    f & \mapsto & (f_{P_1}, \ldots , f_{P_n})
  \end{array}
\right. .
\end{equation}
\end{defn}

\begin{rem}
  Obviously, the above definition depends on the choices of coordinates on the fibres. However, choosing other systems of coordinates yields another code which is isometric to the first one for the Hamming distance.
Thus, to study the minimum distance of $C_L (X, \L)$, the choice of coordinates on the fibres does not matter. 


\end{rem}

\subsection{The parameters of codes on surfaces}

Let us recall briefly how to estimate the parameters of a code $C_L (X, \L)$.
\begin{itemize}
\item The length is elementary: it is the number $n$ of rational points at which sections of the line bundle are evaluated.
In the present article, we always consider the whole set of rational points of the surface.
\item For the dimension, denote by $S$ the space of global sections of $\L$ vanishing at all the $P_i$'s (this space is in general zero in the following examples). Then, the dimension $k$ of the code is 
$$
k=\dim H^0 (X, \L)- \dim S.
$$ 
\item The minimum distance $d$ is 
$$
d=n-\max \left\{ \sharp V(f)(\F_q)\ |\ f\in H^0 (X, \L)\setminus S \right\},
$$
where $V(f)$ denotes the vanishing locus of $f$.
\end{itemize}

\begin{rem}\label{dimtrick}
  Using the above notations. If one proves that
$$
\max  \left\{ \sharp V(f)(\F_q)\ |\ f\in H^0 (X, \L)\setminus \{0\} \right\} \leq n,
$$
then the evaluation map described in (\ref{evalmap}) is obviously injective and hence $S=\{0\}$ and the dimension of the code is that of $H^0 (X, \L)$. 
\end{rem}

Obviously, for such codes, the only parameter whose computation is hard is the minimum distance.
In general, one only looks for lower bounds. It is worth noting that finding a lower bound for the minimum distance is equivalent with finding an upper bound on the number of rational points of the vanishing locus $V(f)$ of an element $f\in H^0 (X, \L)\setminus S$.
Therefore, bounds on the number of rational points of a curve play a central rule in the present article.

\subsection{Bounds on the number of rational points of curves}

Since the vani\-shing locus $V(f)$ of $f\in H^0 (X, \L)$ is not always smooth and irreducible, the classical Weil bound is not suitable for the present problem.
However, Aubry and Perret's bound is suitable.

\begin{thm}[Aubry Perret \cite{AP2}]\label{AP}
  Let $C$ be a geometrically irreducible curve over $\F_q$ with arithmetical genus $p_C$,
then
$$
|\sharp C(\F_q)-(q+1)|\leq p_C \lfloor 2 \sqrt{q} \rfloor.
$$
\end{thm}

\begin{proof}
Denote by $g_C$ the geometric genus of $C$.
  From \cite{AP2} \S 4.1, we have
$$
|\sharp C (\F_q)-(q+1)| \leq (p_C-g_C)+g_C \lfloor 2\sqrt{q} \rfloor.
$$
Since $g_C\leq p_C$ and $\lfloor 2\sqrt{q} \rfloor \geq 1$, we get the result.
\end{proof}

\begin{rem}
  Notice that a version of Aubry Perret's bound exists for reducible curves in \cite{aubryperret}. However, in what follows, when we treat the reducible case, we work component by component.
\end{rem}


Aubry Perret's bounds are sharp for large values of $q$ but can be largely improved when $q$ is small.
In addition, since we are looking for codes on rational surfaces, most of the curves we will deal with are plane.
For plane curves and small values of $q$, one can use another bound. The following result has been first partially conjectured by Sziklai in \cite{sziklai} and then proved by Homma and Kim in \cite{HommaKim}.

\begin{thm}[Homma Kim \cite{HommaKim}]\label{major0}
Let $d$ be a positive integer and $C$ be a plane curve of degree $d$ without $\F_q$--linear component. Then, 
$$
\sharp C(\F_q) \leq 
    (d-1)q+1
$$ 
except for the case $q=4$, $d=4$ and $C$ is projectively equivalent to the curve
\begin{equation}\label{extremal}
K: \quad x^4 + y^4 + z^4 + x^2 y^2 + y^2 z^2 + z^2 x^2 + x^2 y z + x y^2 z + x y z^2 = 0.
\end{equation}
In the exceptional case above, we have $\sharp C (\F_4)=14$.
\end{thm}

The following corollary of Theorem \ref{major0} has been suggested by a reviewer.

\begin{cor}\label{CorMajor}
  Let $C$ be a plane curve of degree $d$ which is not a union of $d$ lines, then
$$
\sharp C (\F_q)\leq (d-1)q+2.
$$
\end{cor}

\begin{proof}
  If $C$ does not contain any $\F_q$--rational line, then it is a straightforward consequence of Theorem \ref{major0}. Assume that $C$ contains $\F_q$--rational lines and set $C=C_1 \cup C_2$, where $C_2$ does not contain any $\F_q$--rational line and $C_1$ is a union of $\F_q$--rational lines.
Set $r:=\deg (C_1)$. By assumption on $C$, we have $r<d$.
From \cite{lettre}, we have $\sharp C_1 (\F_q) \leq rq+1$ and, if $C_2$ does not correspond to the exceptional case of Theorem \ref{major0}, then $\sharp C_2 (\F_q) \leq (d-r-1)q+1$ and we get the result using Theorem \ref{major0}.

In the exceptional case: $q=4$ and $C_2$ is projectively equivalent to the curve $K$ described in (\ref{extremal}). One checks easily that $K(\F_4)=\P^2 (\F_4)\setminus \P^2 (\F_2)$. Therefore, each $\F_4$--rational line meets $C_2$ at least at one $\F_4$--rational point. 
Therefore, the inequality holds in the exceptional case.
\end{proof}

\section{Codes on an elliptic quadric surface}\label{quad}

In this section, we study codes on elliptic quadric surfaces. We refer the reader \cite{Hirschfeld} Part IV, Table 15.4 and \S 15.3.{\sc ii} for a definition of an \emph{elliptic quadric} and for the basic properties of this surface. The aim of this study is first to estimate the parameters of such codes and second to motivate Section \ref{const} in which other rational surfaces yielding good codes are constructed.

\subsection{Previous works on the topic}
Codes of the form $C_L (X, \O_X (2))$ on arbitrary dimensional quadric varieties are first considered by Aubry in \cite{aubry}.
Afterwards, the more specific case of codes $C_L (X, \O_X (2))$ on quadric surfaces is studied in depth by Edoukou in \cite{fred2}.
In both works, it appears that elliptic quadrics turn out to be the ones which provide the best codes in terms of parameters.
However, as far as we know, there does not exist any work on the topic using the property of rationality of these varieties.

\subsection{Context and notations}
In this section, we present a new approach for the study of codes on smooth elliptic quadrics and state a lower bound for the minimum distance of the code $C_L(X, \O_X (3))$. This approach is based on the fact that a smooth quadric in $\P^3$ can be obtained by blowing up $\P^2$ at $2$ points and then by blowing down the resulting surface along a line.

\subsubsection{Construction of quadrics from the projective plane}
Let $P$ denote a closed point of degree $2$ of $\P^2$.
After a base field extension, $P$ splits in two conjugated points $p$ and ${p}^{\varphi}$ defined over $\F_{q^2}$, where $\varphi$ denotes the Frobenius map.
We denote by $L$ the unique rational line of $\P^2$ containing $P$.
The surface $\widetilde{X}$ is the surface obtained by blowing up $\P^2$ at $P$.
The blow up map is denoted by $\pi: \widetilde{X} \rightarrow \P^2$.
We denote by $\widetilde{L}$ the strict transform of $L$ by $\pi$ and by $E$ the exceptional divisor.
Over $\F_{q^2}$, the divisor $E$ splits into a union of two conjugated lines $e$ and ${e}^{\varphi}$.
On $\widetilde{X}$, we have $\widetilde{L}^2=-1$ and hence, by Castelnuovo's criterion (x\cite{H} Chapter V, Theorem 5.7), this curve is the exceptional divisor of some blow up map.
Finally, the surface $X$ obtained by blowing down $\widetilde{X}$ at $\widetilde{L}$ is isomorphic to an elliptic quadric of $\P^3$. We denote by $\psi: \widetilde{X} \rightarrow X$ this blow down map and by $Q$ and $H$ the respective images of $\widetilde{L}$ and $E$ by $\psi$.
The divisor $H$ is prime but splits over $\F_{q^2}$ into a pair of conjugated lines denoted by $h$ and ${h}^{\varphi}$. 
$$\xymatrix{\relax
\widetilde{X} \ar[d]_{\pi} \ar[r]^{\psi} & X \\
\P^2 \ar@{.>}[ur] & }$$

Figure \ref{fig} summarises the above described notations.

\begin{figure}[h]
  \centering
\noindent \xymatrix{\includegraphics[scale=.5]{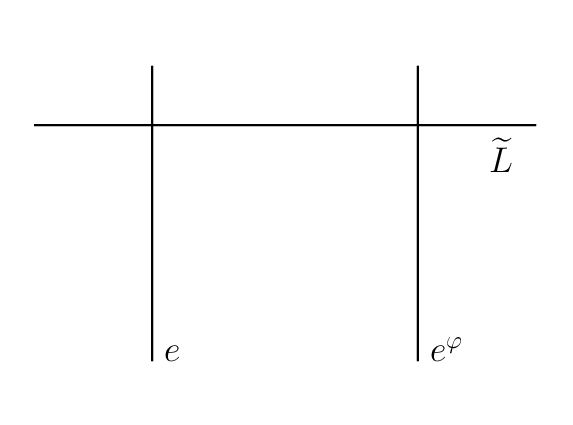} \ar[rr]^{\psi} \ar[dd]^{\pi} & & \includegraphics[scale=.5]{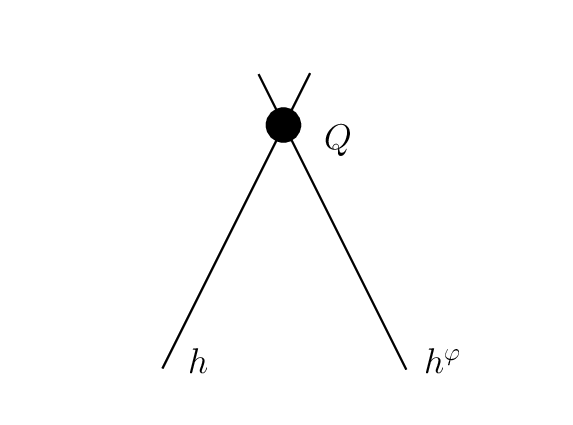} \\ & & \\ \includegraphics[scale=.5]{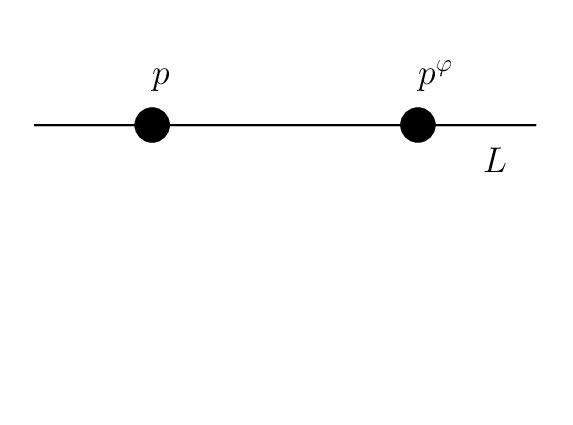} \ar@{.>}[uurr] & & }  
\vspace{-2cm}
  \caption{Illustration of the construction of $X$ from $\P^2$.}
  \label{fig}
\end{figure}

\noindent Now, let us summarise some properties of the involved surfaces

\medbreak

\noindent \textbf{Summary}

\begin{itemize}
\item About $\P^2$
  \begin{enumerate}[(A)]
    \item $\Pic_{\F_q}(\P^2)\cong \mathbf{Z} L \quad \textrm{and}\quad L^2=1.$
  \end{enumerate}
\item About $\widetilde{X}$. 
 
\begin{enumerate}[(A)]
  \setcounter{enumi}{1}
\item\label{inter}
$\Pic_{\F_q}(\widetilde{X})\cong \Z E \oplus \Z \widetilde{L}\quad \textrm{and}\quad  E^2=-2,\quad E.\widetilde{L}=2,\quad \widetilde{L}^2=-1.$

\item\label{L}
 $\pi^{\star}L = \widetilde{L}+E. $
\item $E=e+{e}^{\varphi}$.
\end{enumerate}

\item About $X$

  \begin{enumerate}[(A)]
   \setcounter{enumi}{4}
  \item $H$ corresponds to the cut out of $X$ by its tangent plane at $Q$.
  \item $
\Pic_{\F_q}(X)\cong \Z H, \ \textrm{with} \ H^2=2.
$
  \item $\psi^{\star}H = 2\widetilde{L}+E$.
  \item $H=h+{h}^{\varphi}$.
  \end{enumerate}

\end{itemize}

To estimate the minimum distance of functional codes on an elliptic quadric, the two following lemmas are useful.

\begin{lem}\label{plane_rep}
  Let $D$ be an effective divisor on $X$ containing $Q$ and which is smooth at this point.
  Let $s$ be the positive integer such that $D \sim sH$.
  Let $\widetilde{D}$ be the strict transform of $D$ by $\psi$ and $D'$ be the image of $\widetilde{D}$ by $\pi$.
Then,
\begin{enumerate}[(i)]

\item\label{tchou} $\sharp D(\F_q)=\sharp D'(\F_q)$; 
\item\label{v} $D'$ is singular at $P$ with multiplicity $s-1$;
\item\label{u} $D'$ has degree $2s-1$.
\end{enumerate}
Figure \ref{pict} illustrates the case $s=3$.
\end{lem}

\begin{figure}[h]
  \centering
\xymatrix{\includegraphics[scale=.5]{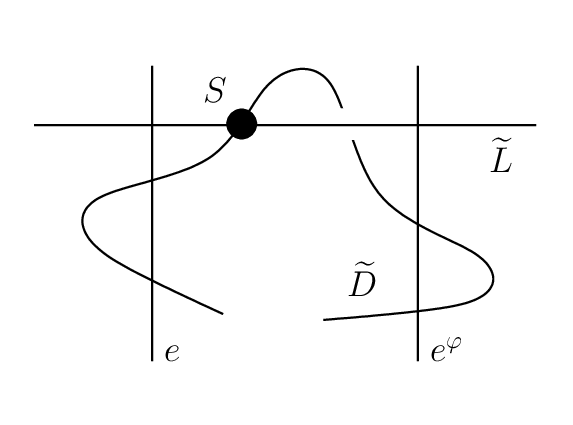} \ar[rr]^{\psi} \ar[dd]_{\pi} & & \includegraphics[scale=.5]{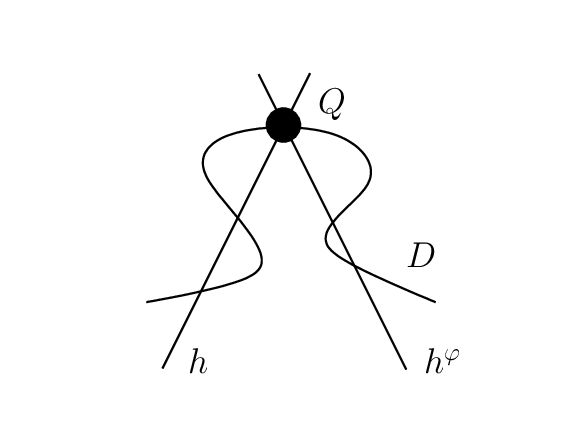} \\
 & & \\
\includegraphics[scale=.5]{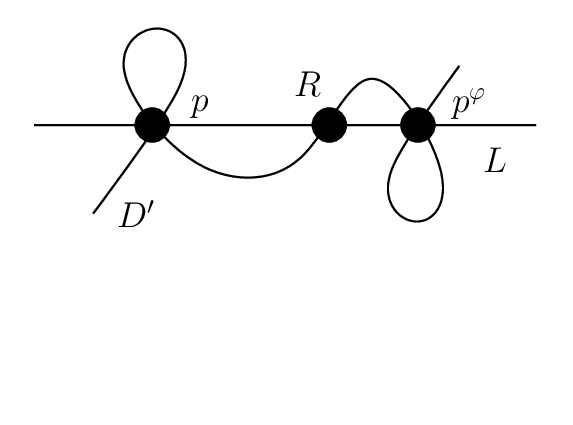} \ar@{.>}[rruu] & & }  
\vspace{-2cm}
  \caption{The divisors $D, \widetilde{D}$ and $D'$ for $s=3$.}
\label{pict}
\end{figure}

\begin{rem}
  The assertion ``$D$ is effective and $D\sim sH$'' is equivalent to ``$D$ is a cut out of $X$ by a surface of degree $s$ which has no common component with $X$''.
\end{rem}

\begin{proof}[Proof of Lemma \ref{plane_rep}]
Since $\pi$ consists in blowing up $\P^2$ at a nonrational closed point, it has no influence on the number of rational points. Thus $\sharp D'(\F_q)=\sharp \widetilde{D}(\F_q)$. Moreover, since $D$ is smooth at $Q$, we have $\sharp D(\F_q)=\sharp \widetilde{D}(\F_q)$. This proves (\ref{tchou}).

Recall that $H$ denotes the intersection divisor of $X$ by its tangent plane at $Q$ and that $\psi^{\star}H = 2\widetilde{L}+E$. Therefore, the strict transform of $H$ by $\psi$ equals $E$.
Since $D$ contains $Q$ and is smooth at it, it has a rational tangent line at $Q$. Moreover, since $h$ and ${h}^{\varphi}$ are not defined over $\F_q$, then $D$ meets $h$ and ${h}^{\varphi}$ transversally at $Q$.
Since $D \sim sH$, it meets $h$ (resp. ${h}^{\varphi}$) at $s-1$ geometric points (counted with multiplicities) out of $Q$.

Therefore, $\widetilde{D}$ meets $e$ (resp. $e^{\varphi}$) at $s-1$ geometric points counted with multiplicities. This gives
\begin{equation}
  \label{DE}
  \widetilde{D}.E=2(s-1)
\end{equation}

\noindent Moreover, after contracting $e$ and $e^{\varphi}$ (i.e. applying $\pi$), the image $D'$ of $\widetilde{D}$ is singular with multiplicity $s-1$ at $p$ and $p^{\varphi}$, that is at $P$. This proves (\ref{v}).

\medbreak

Finally, since $D$ is smooth at $Q$, we have
\begin{equation}
  \label{Dtilde}
  \psi^{\star} D=\widetilde{D}+\widetilde{L} \quad \textrm{and} \quad \widetilde{D}.\widetilde{L}=1. 
\end{equation}
Indeed, recall that $\widetilde{L}$ is the exceptional divisor of $\psi$.
Moreover, since $D'$ is the image of $\widetilde{D}$ by $\pi$ and since $\widetilde{D}$ and $E$ have no common component, then $\widetilde{D}$ is also the strict transform of $D'$ by $\pi$. Thus, since it has already been proved that $D'$ has multiplicity $s-1$ at $P$, we get
\begin{equation}
  \label{Dprime}
  \pi^{\star}D'=\widetilde{D}+(s-1)E.
\end{equation}

\noindent Since the degree of $D'$ equals the intersection product $D'.L$, using (\ref{inter}), (\ref{L}), (\ref{DE}), (\ref{Dtilde}), (\ref{Dprime}) and \cite{H} Chapter V, Proposition 3.2(a), we get
$$
\begin{array}{rcl}
  D'.L=\pi^{\star}D'.\pi^{\star}L & = & \widetilde{D}.\widetilde{L}+(s-1)E.\widetilde{L}+\widetilde{D}.E+(s-1)E^2\\
 & = & 1+2(s-1)+2(s-1)-2(s-1)=2s-1,
\end{array}
$$
which proves (\ref{u}).
\end{proof}

In our particular case, the following Proposition gives a sharper bound than that of Homma and Kim (Theorem \ref{major0}).

\begin{prop}\label{major}
Let $s$ be an integer such that $s\geq 2$ and $D \subset X$ be an $\F_q$--irreducible curve such that $D\sim sH$. Then 
$$
\sharp D(\F_q) \leq q(2s-2).
$$ 
\end{prop}

\begin{proof}
\textit{Step 1.} Assume that $D$ has at least one nonsingular rational point.
Recall that the automorphism group of an elliptic quadric acts transitively on its set of rational points (\cite{Hirschfeld} Part IV, Theorem 15.3.19).
Thus, after applying a suitable automorphism, one can assume that $D$ contains $Q$ and is smooth at it.
From Lemma \ref{plane_rep}, there exists a plane curve $D'$ of degree $2s-1$ which is singular with multiplicity $s-1$ at $P$ (which has degree $2$).
Moreover $\sharp D'(\F_q)=\sharp D(\F_q)$.
Therefore, as illustrated by Figure \ref{pict}, the line $L$ containing $P$ meets $D'$ at a unique other geometric point $R$. This point $R$ is thus rational and smooth (it is actually the image by $\pi$ of the preimage $S$ of $Q$ by $\psi_{|\widetilde{D}}: \widetilde{D} \rightarrow D$).

Now, consider the linear system of lines containing $R$. This linear system has $(q+1)$ rational elements
 $L_1, \ldots , L_{q+1}$ which cover all the rational points of $\P^2$.
Among the $L_i$'s, one finds the line $L$ which meets $D'$ only at $P$ and $R$ and hence meets $D'$ at only one rational point (the point $R$).
Since $D'$ is smooth at $R$, the tangent $T_R D'$ to $D'$ at $R$ is rational and hence is one of the $L_i$'s. Moreover, a simple argument based on B\'ezout's Theorem proves that $T_R D' \neq L$. 
Finally, we get
$$
\sharp D'(\F_q)\leq q(2s-2).
$$ 

\medbreak

\noindent \textit{Step 2.} If all the rational points of $D$ are singular, then, from Lemma \ref{TotSingular} below, $\sharp D(\F_q)\leq s(s+1)-2q$.
There remains to check that $s(q+1)-2q \leq q(2s-2)$ for all $s\geq 2$ and $q\geq 2$, which is elementary.
\end{proof}

\begin{lem}\label{TotSingular}
Let $s$ be a positive integer and $D\subset X$ be an $\F_q$--irreducible curve such that $D\sim sH$. Assume moreover that the rational points of $D$ are all singular. Then
$$
\sharp D (\F_q) \leq \left\{
\begin{array}{ccc}
1 & \textrm{if} & s=1 \\
s(q+1)-2q & \textrm{if} & s\geq 2 \end{array} \right. . 
$$ 
\end{lem}

\begin{proof}
If $s=1$, then, $D$ is a cut out of $X$ by a plane. Thus, $D$ is an irreducible plane conic. Since it is assumed to be singular, it is a union of two conjugated lines meeting at a single point which is the only rational point of $D$.

Now, assume that $s\geq 2$.
Choose two distinct rational points $A,B$ of $D$ (if they do not exist, then $\sharp D(\F_q)$ satisfies obviously the upper bound). Consider the set of $q+1$ rational plane cut outs $H_1, \ldots , H_{q+1}$ of $X$ containing $A$ and $B$. These plane cut outs cover all the rational points of $X$ and each one of them contains $A$ and $B$. Using that $D$ is singular at all of its rational points, we obtain
$$
\begin{array}{cccc}
& D.(H_1+\cdots +H_{q+1}) & \geq & 2 \sharp (D(\F_q)\setminus \{A,B\}) + 2(q+1)\sharp \{A,B\} \\
\Rightarrow & 2s(q+1) & \geq & 2(\sharp D(\F_q)-2)+4(q+1) \\
\Rightarrow & s(q+1)-2q & \geq & \sharp D(\F_q). 
\end{array}
$$
  
\end{proof}

\subsection{Application to the study of $C_L (X, \O_X(3))$}

For a fixed base field $\F_q$, the code $C_L (X, \O_X (3))$ has length
$n=q^2+1$, which is the number of rational points of $X$ (\cite{Hirschfeld} Part IV, Table 15.4). 
To compute the dimension and the minimum distance of this code, we need Lemma \ref{dimH0} and Proposition \ref{cube} below.
The parameters of this code are summarised further in Theorem \ref{ParametersCodesQuadric}.

\begin{lem}\label{dimH0}
Let $m$ be a nonnegative integer.
The dimension of $H^0(X, \O_X (m))$ is $(m+1)^2$.  
\end{lem}

\begin{proof}
Let $F$ be a homogeneous polynomial of degree $2$ such that $F(x,y,z,t)=0$ is an equation of $X$. The space $H^0 (X, \O_X (m))$ corresponds to the space of homogeneous forms of degree $m$ modulo the forms vanishing on $X$, that is the multiples of $F$. 
Thus, we have the isomorphism
$$
H^0(X, \O_X (m))\cong H^0 (\P^3, \O_{\P^3}(m))/H^0(\P^3, \O_{\P^3}(m-2)).F,
$$
which entails
$$
\begin{array}{ccl}
  \dim H^0(X, \O_X (m)) & = & \dim H^0 (\P^3, \O_{\P^3}(m)) - \dim
  H^0(\P^3, \O_{\P^3}(m-2)) \\
   & = & { \left(
  \begin{array}{c}
    m+3 \\ 3
  \end{array}
\right)} -
{\left(
  \begin{array}{c}
    m+1 \\ 3
  \end{array}
\right)} =(m+1)^2.
\end{array} 
$$
\end{proof}

\begin{rem}
Lemma \ref{dimH0} entails that the dimension of $C_L (X, \O_X (3))$ has dimension at most $16$.
  Since we must have $n\geq k \geq 16$ and $n=q^2+1$, the study of such codes makes sense only for $q\geq 4$. It starts to be interesting for $q\geq 5$.
Therefore in the following statements, we assume that $q\geq 5$.
\end{rem}

\begin{prop}\label{cube}
Assume that $q\geq 5$. 
Let $C$ be an effective divisor on $X$ such that $C\sim 3H$. Then,
$$
\sharp C(\F_q) \leq \max (3q+3,\ \min(4q,\ q+1+4 \lfloor 2\sqrt{q} \rfloor)).
$$
\end{prop}

\begin{proof}
Using that $\Pic_{\F_q} (X)$ is generated by $H$,
 we separate the proof in three cases:
  \begin{enumerate}[(i)]
  \item\label{r3} $C=C_1 \cup C_2 \cup C_3$, where the $C_i$'s are $\F_q$--irreducible and are all three linearly equivalent to $H$;
  \item\label{r2} $C=C_1\cup C_2$, where $C_1, C_2$ are $\F_q$--irreducible and $C_1 \sim H$ and $C_2 \sim 2H$;
  \item\label{irr} $C$ is $\F_q$--irreducible.
  \end{enumerate}

To treat these distinct cases, we need to compute the arithmetical genus of a geometrically irreducible (possibly singular) curves embedded in $X$. For that, we use the adjunction formula (\cite{ser} Chapter IV \S 2 Proposition 5) asserting that the arithmetical genus of a geometrically irreducible curve (possibly singular) $C$ embedded in $X$ is
$$
p_a (C)=1+\frac 1 2 C.(K+C),
$$
where $K$ denotes the canonical class of $X$. From \cite{H} Chapter II, Example 8.20.3, we get $K\sim -2H$. Therefore, if $C$ is a geometrically irreducible curve embedded in $X$, we get
\begin{equation}
  \label{ArithGenus}
 C\sim aH \ \Longrightarrow \ p_a(C)=1+a(a-2).
\end{equation}

The case (\ref{r3}) is elementary, in this situation $C$ is a union of $3$ plane cut outs of $X$. Such cut outs are plane $\F_q$--irreducible conics and hence have either $1$ (a pair of conjugated lines) or $q+1$ (a smooth plane conic) rational points. Thus, in situation (\ref{r3}), $\sharp C(\F_q)\leq 3q+3$.

In situation (\ref{r2}), as in the previous case we have $\sharp C_1 \leq q+1$.
If $C_2$ is not geometrically irreducible, then its rational points are singular (they lie at the intersection of irreducible components defined over $\overline{\F}_q$). Therefore, from Lemma \ref{TotSingular}, we get $\sharp C_2 \leq 2$.
Now, if $C_2$ is geometrically irreducible, then, using (\ref{ArithGenus}), one proves that $p_a(C_2)=1$ and from Aubry and Perret's bound, $\sharp C_2(\F_q)\leq q+1+\lfloor 2\sqrt{q}\rfloor$. An easy computation proves that $q+1+\lfloor 2\sqrt{q}\rfloor \leq 2q+2$ for all $q\geq 2$. Thus, we also have $\sharp C(\F_q)\leq 3q+3$.

In case (\ref{irr}), if $C$ is not geometrically irreducible, then, as in the previous case, one proves that $\sharp C(\F_q) \leq q+2$ by using Lemma \ref{TotSingular}.
If it is geometrically irreducible, then
using (\ref{ArithGenus}), one proves that $p_a(C)=4$ and from Proposition \ref{major} together with Theorem \ref{AP}, we get $\sharp C (\F_q)\leq \min (4q, q+1+4\lfloor 2\sqrt{q}\rfloor)$.
\end{proof}

Finally, we are able to estimate the parameters of the code $C_L (X, \O_X (3))$. This is the purpose of the following theorem.

\begin{thm}\label{ParametersCodesQuadric}
Let $X$ be an elliptic quadric over $\F_q$ with $q\geq 5$.
The code $C_L (X,\O_X (3))$ has parameters $[q^2+1, 16, \geq \delta]$, where
$$
\delta = q^2+1-\max (3q+3,\ \min (4q,\ q+1+4 \lfloor 2\sqrt{q} \rfloor)).
$$
That is:
$$
\delta=\left\{
  \begin{array}{ccc}
    q^2+1-4q & \textrm{if} & q\leq 7\\
    q^2-q- 4\lfloor 2\sqrt{q} \rfloor & \textrm{if} & 8 \leq q \leq 13 \\
    q^2-2-3q  & \textrm{if} & q\geq 16.
  \end{array}
\right.
$$
\end{thm}

\begin{proof}
  The length has already been computed above.
For the minimum distance, it is a straightforward consequence of Proposition \ref{cube}.
The dimension is a straightforward consequence of Lemma \ref{dimH0} together with Remark \ref{dimtrick}.
\end{proof}

Table \ref{CLXO3} gives the parameters of such codes for small values of $q$. In addition, the lower bound for the minimum distance is compared with the best known minimum distance for the same length and dimension.
It shows that codes of the form $C_L (X, \O_X (3))$ are good compared to the table of the best known codes \cite{grassl} and \cite{minT2}.

\medbreak

\begin{table}[!h]
\begin{tabular}{|c|c|c|c|c|}
\hline
$q$ & $n$ & $k$ & $d$ &  Best $d$ \\ 
 & & & &  up to now \\
\hline 
5 & 26 & 16 & $\geq$ 6 & 8 \\
\hline
7 & 50 & 16 & $\geq$ 22 & 26 \\
\hline
8 & 65 & 16 & $\geq$ 36 & 38 \\
\hline
9 & 82 & 16 & $\geq$ 48 & 52  \\
\hline
\end{tabular}
\medbreak
\caption{Parameters of $C_L(X, \O_X (3))$, when $X$ is an elliptic quadric.}
\label{CLXO3}
\end{table}
\medbreak


\subsection{A remark about the study of $C_L (X, \O_X (2))$}
The code $C_L (X, \O_X (2))$ is studied in \cite{fred2} when $X$ is a quadric of any kind. However, it is interesting to note that the elliptic case can be easily obtained from our work. Using the previous methods, one gets the following proposition which corresponds to \cite{fred2} Proposition 6.6.

\begin{prop}
  Let $X'$ be a quadric surface distinct from $X$, let $C$ be the intersection of $X$ and $X'$, then
$$
\sharp C(\F_q)\leq 2q+2
$$
and this upper bound is reached.
Thus, the parameters of the code $C_L (X, \O_X (2))$ are $[q^2+1, 9 , q^2-2q-1]$.
\end{prop}

\begin{proof}
  Two cases must be considered:
  \begin{enumerate}[(i)]
  \item\label{51} $C$ is a union of two plane cut outs;
  \item\label{52} $C$ is $\F_q$--irreducible.
  \end{enumerate}
Case (\ref{51}) yields $\sharp C(\F_q)\leq 2q+2$ and this upper bound is reached when both plane cut outs have $q+1$ rational points and do not meet at rational points. Case (\ref{52}) yields $\sharp C(\F_q)\leq q+1+\lfloor 2\sqrt{q} \rfloor$ from Theorem \ref{AP} which is smaller than $2q+2$ for all $q$. 
\end{proof}





\section{Constructions of rational surfaces yielding good codes}\label{const}

Consider the case of the code $C_L (X, \O_X (n))$ on an elliptic quadric. By the blow up and blow down operation, the linear system associated to $\O_X (n)$ on $X$ corresponds to a linear system in $\P^2$ having the closed point $P$ as a base point. Therefore, such curves defined over $\F_q$ cannot contain any rational line of $\P^2$ but $L$ whose strict transform is contracted. 
Thus, the elements of the linear system cannot have \emph{too many} $\F_q$--irreducible components.

This is the motivation of the following examples. We will give some particular linear systems of $\P^2$ whose $\F_q$--rational elements cannot break into too many $\F_q$--irreducible components and compute the maximal number of rational points of the  elements of the linear system. Such a linear system provides a line bundle $\L$ over a rational surface $X$ obtained from $\P^2$ after some possible blow ups and blow downs. The parameters of the code $C_L (X, \L)$ on this surface arise from the properties of the linear system.

\subsection{The projective plane blown up at a rational point and a point of degree $4$}

\subsubsection{Context}
Consider the projective plane $\P^2$ and let $P$ be a rational point.
Denote by $\varphi$ the Frobenius map.
Let $l$ and $l^{\varphi}$ be a pair of conjugated lines defined over $\F_{q^2}$ and meeting at $P$.
Denote by $D$ the $\F_q$--rational conic $D:=l\cup l^{\varphi}$.
Let $R$ be a closed point of degree $4$ of $D$.
Over $\F_{q^4}$, this point splits into $4$ points $r, r^{\varphi}, r^{\varphi^2}$ and $r^{\varphi^3}$, where $\varphi$ denotes the Frobenius map.
The following picture illustrates this context.
\begin{center}
  \includegraphics[scale=0.7]{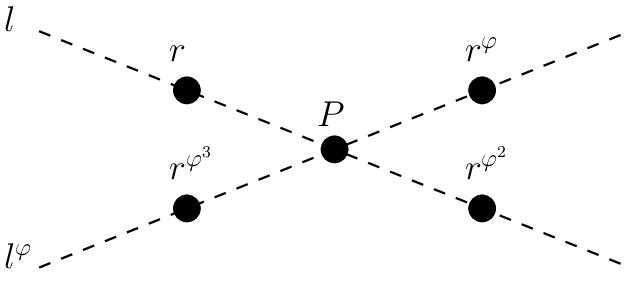}
\end{center}

\begin{defn}[The surface $Y$]
Let $Y$ be the surface obtained from $\P^2$ by blowing up $P$ and $R$.  
We denote by $\pi: Y \rightarrow \P^2$ the blow up map and by $E$ and $F$ the exceptional divisors above $P$ and $R$ respectively.
\end{defn}

\begin{defn}[The line bundle $\mathcal{F}_i$]
  Let $i\geq 4$ be an integer.
Let $\Lambda_i$ be the linear system of plane curves of degree $i$ containing $R$ with multiplicity at least $1$ and $P$ with multiplicity at least $2$.
Let $\mathcal{F}_i$  be the line bundle over $Y$ associated to the linear system $\pi^{\star} \Lambda_i -2E-F$. 
\end{defn}

\begin{rem}
  The linear system $\pi^{\star} \Lambda_i -2E-F$ is base point free for all $i\geq 4$ and very ample for $i\geq 5$ (use \cite{H} Chapter II, Remark 7.8.2).
\end{rem}

\subsubsection{The code $C_L (Y, \mathcal{F}_4)$}\label{sectruc}

\begin{thm}
The parameters of the code $C_L (Y, \mathcal{F}_4)$ are 
$$
[(q+1)^2, 8, q^2-q-2].
$$
\end{thm}

\begin{proof}
The code has length $n=\sharp Y (\F_q)=(q+1)^2$.

For the dimension, we need to know the dimension of the linear system $\Lambda_4$.
The dimension of the linear system of plane quartics is $14$.
The interpolating condition at $P$ imposes $3$ constraints and the vanishing condition at $R$ imposes $4$ other constraints.
These $7$ constraints can be proved to be independent (details are left to the reader) and hence the dimension of $\Lambda_4$ is $7$ and that
 of $H^0 (Y, \mathcal{F}_i)$ is $8$. Using Remark \ref{dimtrick} together with Proposition \ref{LPB5} below, we see that the dimension of the code is also $8$.

The minimum distance $d$ is given by Proposition \ref{LPB5}.  
\end{proof}

\noindent \textbf{Caution.}
This example is pretty different from the former one since here a divisor $C\in \Lambda_i$ and the divisor $C' :=\pi^{\star} C -2E-F$ have not always the same number of rational points. Indeed, from $C$ to $C'$, the point $P$ may ``split'' into two distinct rational points or into a closed point of degree $2$.
Moreover, if $C$ has multiplicity $\geq 3$ at $P$, then $C'$ contains the whole curve $E$.

\begin{prop}\label{LPB5}
  Let $C$ be a curve in the linear system $\pi^{\star}\Lambda_4 -2E-F$, then
$$
\sharp C (\F_q)\leq 3q+3
$$
and the bound is reached.
\end{prop}

\begin{proof}
Let $B$ be the plane curve corresponding to $C$ in $\Lambda_4$ (i.e. $B=\pi (C)$).
We separate the proof in four distinct cases.

\begin{enumerate}[(i)]
    \item\label{LBP51}   $B=B_1 \cup B_2$, where $B_1 $ is an $\F_q$--irreducible conic containing $R$ and avoiding $P$ and $B_2$ is a conic which is singular at $P$.

    \item\label{LBP52}
          $B=B_1 \cup B_2$, where $B_1 $ is an $\F_q$--irreducible conic containing $P$ and $R$ (notice that in this situation $B_1=l\cup l^{\varphi}$ and hence is singular at $P$) and $B_2$ is an arbitrary conic.
  
    \item\label{LBP53} $B=B_1 \cup B_2$, where $B_1 $ is an $\F_q$--irreducible cubic containing $R$ and $P$ and $B_2$ is a line containing $P$.

    \item\label{LBP54} $B$ is an $\F_q$--irreducible quartic containing $R$ and singular with multiplicity $2$ at $P$.
  \end{enumerate}

\noindent The four distinct situations are illustrated by the following pictures.

\begin{center}
  \begin{tabular}{cc}
    \includegraphics[scale=.5]{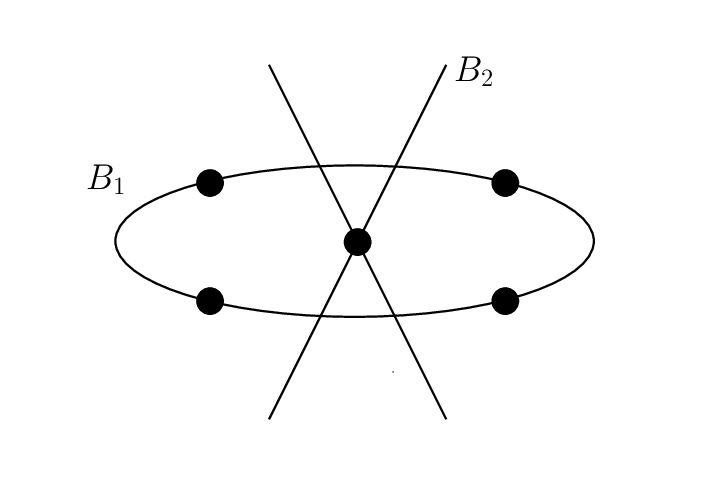} & \includegraphics[scale=.5]{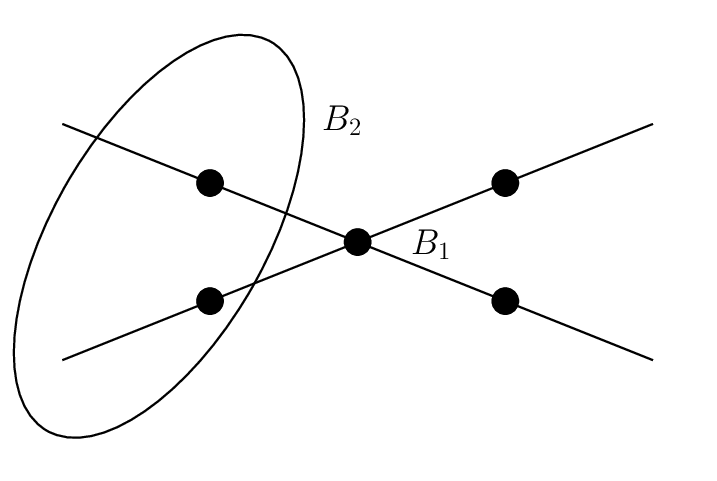} \\
(\ref{LBP51}) & (\ref{LBP52}) 
 \end{tabular}
\end{center}

\begin{center}
 \begin{tabular}{cc}
\includegraphics[scale=.5]{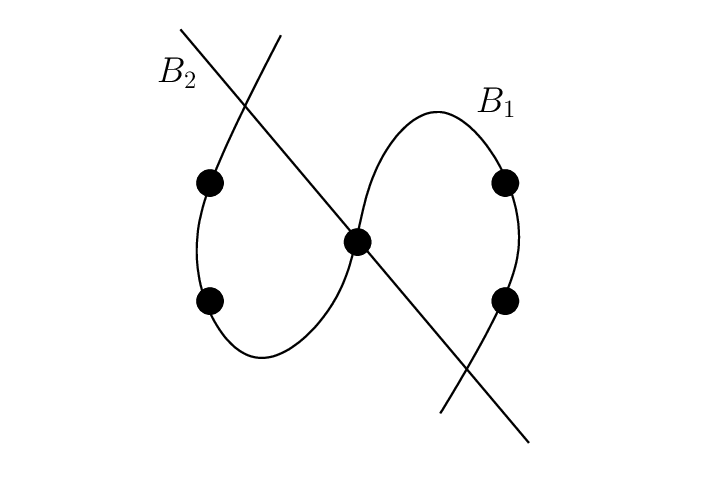} &
\includegraphics[scale=.5]{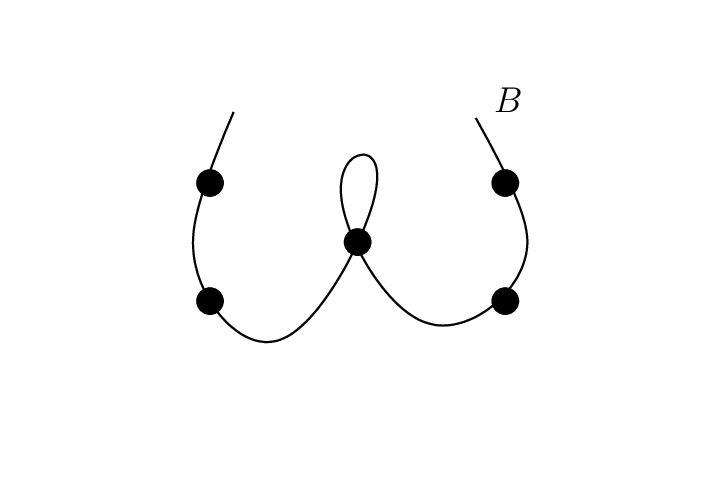} \\
(\ref{LBP53}) & (\ref{LBP54})  
  \end{tabular}
\end{center}

Let us make a few remarks about these distinct cases in order to make sure they are the only possible ones.
First, notice that for case (\ref{LBP53}) if $B_1$ is a cubic, then it must contain $P$ since $B_2$ is a line and hence cannot be singular at $P$. Moreover if $B_1$ is singular at $P$, then, from B\'ezout's Theorem, it would contain $l$ and $l^{\varphi}$ and hence would not be $\F_q$--irreducible. Thus, $B_1$ must be smooth at $P$ and hence $B_2$ must contain $P$.
This situation is interesting since in this case the multiplicity of $B$ at $P$ cannot be $\geq 3$ and hence $C$ cannot contain $E$.
By the same manner in case (\ref{LBP54}), the curve $B$ cannot be singular with multiplicity $>2$ at $P$.

\medbreak

Now let us treat these distinct cases.
In case (\ref{LBP51}), the worst situation is when $B_1$ is smooth and $B_2$ is a union of two rational lines containing $P$ and which do not meet $B_1$ at rational points. Then $C=\widetilde{B}_1 + \widetilde{B}_2$. The curve $\widetilde{B}_2$ is union of two skew lines, thus $\sharp \widetilde{B}_2 (\F_q)=2q+2$ and the curve $\widetilde{B}_1$ is isomorphic to $B_1$.
Thus, $\sharp C (\F_q)\leq 3q+3$ and this upper bound is reached since the worst case happens for some $C$.

In case (\ref{LBP52}), the curve $B_1$ equals $D=l\cup l^{\varphi}$. The worst situation is when $B_2$ is a pair of rational lines containing $P$. In this situation
$$
C=  \widetilde{B}_1 \cup \widetilde{B}_2 \cup E.
$$
The curve $\widetilde{B}_1$ is a union of two skew conjugated lines over $\F_{q^2}$ and hence has no rational points.
Thus, $\sharp C (\F_q)\leq 3q+1$.

In case (\ref{LBP53}), we have $C=\widetilde{B}_1 + \widetilde{B}_2$ and the components are respectively isomorphic to $B_1$ and $B_2$.
Thus, applying Corollary \ref{CorMajor} to each irreducible component, we get $\sharp C (\F_q)\leq 3q+ 3$.

In case (\ref{LBP54}), from Corollary \ref{CorMajor}, we have $\sharp B(\F_q)\leq 3q+2$.
Moreover,
as noticed before, $B$ has multiplicity exactly $2$ at $P$, then $C=\widetilde{B}$ and $C$ contains at most $2$ rational points above $P$.
Thus, $\sharp C(\F_q)\leq 3q+3$.
\end{proof}

Table \ref{Tfin} gives the parameters of the code $C_L (Y, \mathcal{F}_4)$ for small values of $q$. In the right column, the minimum distance of the best known code for the same length and dimension is given. This shows that these codes are good.

\begin{table}[!h]
  \centering
  \begin{tabular}{|c|c|c|c|c|}
    \hline
$q$ & $n$ & $k$ & $d$ & Best $d$ \\
 &  &  &  &  up to now \\
    \hline
  3 & 16 & 8 & 4 & 6 \\
\hline
  4 & 25 & 8 & 10 & 12 \\
\hline 
  5 & 36 & 8 & 18 & 21 \\
\hline 
  7 & 64 & 8 & 40 & 41 \\
\hline 
 8 & 81 & 8 & 54 & 58 \\
\hline 
9 & 100 & 8 & 70 & 75 \\
\hline 
  \end{tabular}
\medbreak
 \caption{Parameters of the code $C_L (Y, \mathcal{F}_4)$.}
  \label{Tfin}
\end{table}

\subsection{The projective plane blown up at a point of degree $3$}

\subsubsection{Context} Consider the projective plane and a closed point $P$ of degree $3$ which is not contained in any rational line. After a base field extension, $P$ splits into three non collinear points $p, p^{\varphi}$ and $p^{\varphi^2}$, where $\varphi$ denotes the Frobenius map.

\begin{defn}[The surface $Z$]
Let $Z$ be the projective plane blown up at $P$. We denote by $\pi: Z \rightarrow \P^2$ the blow up map and by $E$ the exceptional divisor.  
\end{defn}

\begin{defn}[The line bundles $\mathcal{L}_i$]\label{GamLam}
  Let $i\geq 3$ be an integer. Let $\Gamma_i$ be the linear system of plane curves of degree $i$ containing $P$. We call $\L_i$ the line bundle over $Z$ associated to $\pi^{\star}{\Gamma}_i-E$.
\end{defn}

Let us study some codes on $Z$.

\subsubsection{The code $C_L(Z, \L_3)$}

\begin{thm}
  The parameters of $C_L (Z, \L_3)$ are
$$[q^2+q+1, 7, q^2-q-1].$$
\end{thm}

\begin{proof}
Since $Z$ is obtained from $\P^2$ by blowing up non rational points, it has the same number of rational points as $\P^2$. Thus, the length is
$n=q^2+q+1$. 
The linear system $\Gamma_3$ has dimension $6$ (\cite{H} Chapter V, Corollary 4.4(a)), thus the dimension of the code is $k=7$. 
The minimum distance is given by Proposition \ref{G3} below.  
\end{proof}

\begin{prop}\label{G3}
  Let $C$ be an $\F_q$--rational element of the linear system $\Gamma_3$ (see Definition \ref{GamLam}). Then,
$$
\sharp C (\F_q)\leq 2q+2
$$
and this upper bound is reached.
\end{prop}

\begin{proof}
  Consider the $\F_q$--irreducible components of $C$ containing $P$. Since $P$ has degree $3$ and is not contained  in any rational line, these $\F_q$--irreducible components are either a conic or an $\F_q$--irreducible cubic.
Thus there are two possibilities.
\begin{enumerate}[(i)]
\item\label{G31} $C=C_1 \cup C_2$ where $C_1$ is an $\F_q$--irreducible conic containing $P$ and $C_2$ is a rational line.
\item\label{G32} $C$ is an $\F_q$--irreducible cubic.
\end{enumerate}

\noindent The two distinct cases are illustrated by the pictures below.

In both cases, $C$ is not a union of $\F_q$--rational lines and the upper bound is a straightforward consequence of Corollary \ref{CorMajor}.
In case (\ref{G31}), if $C_2$ does not meet $C_1$ at rational points, then $\sharp C(\F_q)=2q+2$ and hence the bound is reached.
\end{proof}

\begin{center}
  \begin{tabular}{cc}
    \includegraphics[scale=.5]{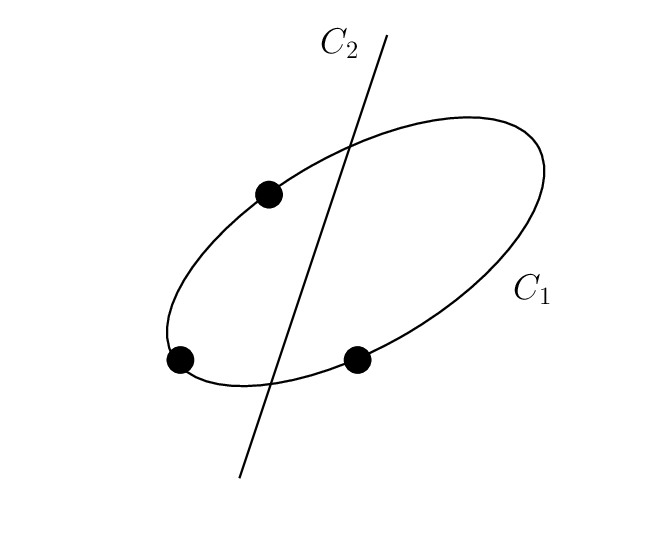} & \includegraphics[scale=.5]{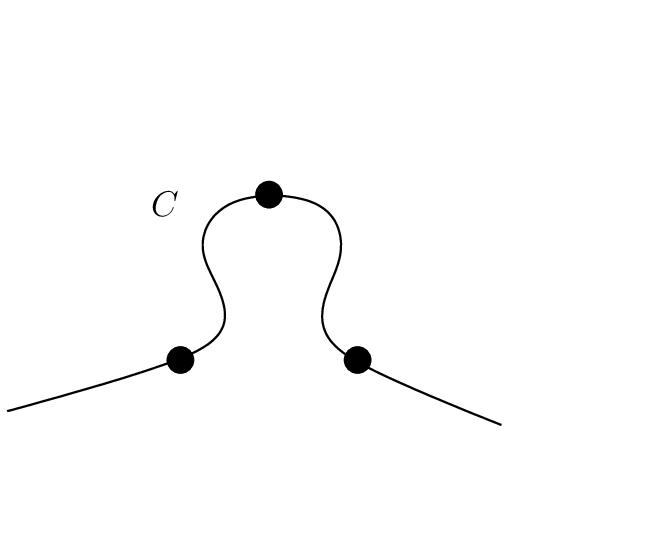} \\
(\ref{G31}) & (\ref{G32})   
  \end{tabular}
\end{center}


Table \ref{TabCL3} gives the parameters of the code $C_L (Z, \L_3)$ for several values of $q$. The right hand column gives the best known minimum distance for these fixed length and dimension. This show that these codes for small values of $q$ are as good as the best known codes.

\begin{table}[!h]
  \centering
  \begin{tabular}{|c|c|c|c|c|}
    \hline
$q$ & $n$ & $k$ & $d$ & Best $d$ \\
    &     &     &     &  up to now \\
    \hline
  3 & 13 & 7 & 5 & 5 \\
\hline
  4 & 21 & 7 & 11 & 11 \\
\hline 
  5 & 31 & 7 & 19 & 19 \\
\hline 
  7 & 57 & 7 & 41 & 41 \\
\hline 
 8 & 73 & 7 & 55 & 55 \\
\hline 
9 & 91 & 7 & 71 & 71 \\
\hline 
  \end{tabular}
\medbreak
 \caption{Parameters of the code $C_L (Z, \L_3)$.}
  \label{TabCL3}
\end{table}

\subsubsection{The code $C_L (Z, \L_4)$}\label{B1}

\begin{thm}
  The parameters of $C_L (Z, \L_4)$ are
$$[q^2+q+1, 12, q^2-2q-1].$$
\end{thm}

\begin{proof}
The length is $n=q^2+q+1$ (as for $C_L (Z, \L_3)$).
The dimension of the linear system $\Gamma_4$ is $11$, since the linear system of plane quartics is $14$ and the vanishing condition at $P$ imposes $3$ independent constraints (details are left to the reader).
Thus, the code has dimension $k=12$. Its minimum distance is given by the following Proposition.  
\end{proof}

\begin{prop}
Assume that $q\geq 4$.
Let $C$ be an $\F_q$--rational element of $\Gamma_4$.
Then,
$$
\sharp C (\F_q) \leq 3q+2
$$
and this upper bound is reached.
\end{prop}

\begin{proof}
  The curve $C$ can be of the form:
  \begin{enumerate}[(i)]
  \item\label{ZL41} $C=C_1 \cup C_2$ where $C_1$ is an $\F_q$--irreducible conic containing the point $P$ and $C_2$ is a conic (possibly reducible);
  \item\label{ZL42} $C=C_1 \cup C_2$ where $C_1$ is an $\F_q$--irreducible cubic containing the point $P$ and $C_2$ is an $\F_q$--rational line;
  \item\label{ZL43} $C$ is an $\F_q$--irreducible quartic.
  \end{enumerate}
The three distinct situations are illustrated by the following pictures.
\begin{center}
  \begin{tabular}{cc}
    \includegraphics[scale=.5]{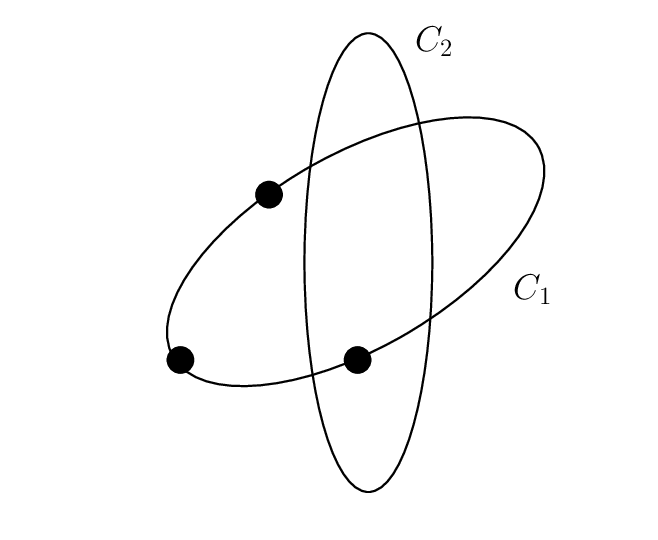} & \includegraphics[scale=.5]{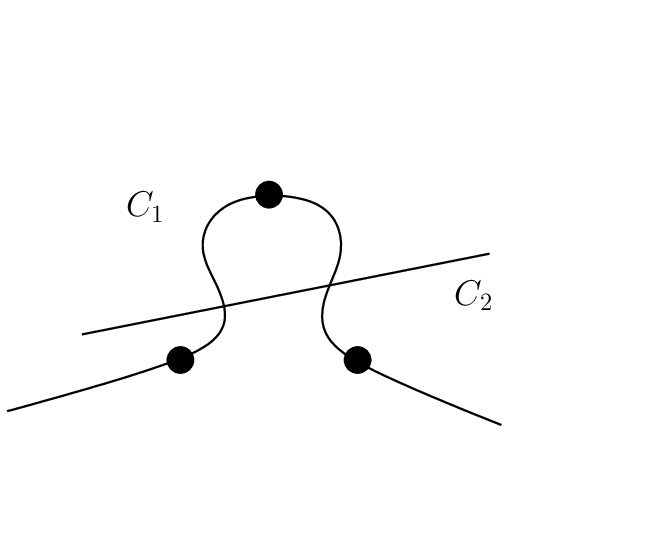} \\
(\ref{ZL41}) & (\ref{ZL42}) 
  \end{tabular}
\end{center}

\begin{center}
  \begin{tabular}{c}
    \includegraphics[scale=.5]{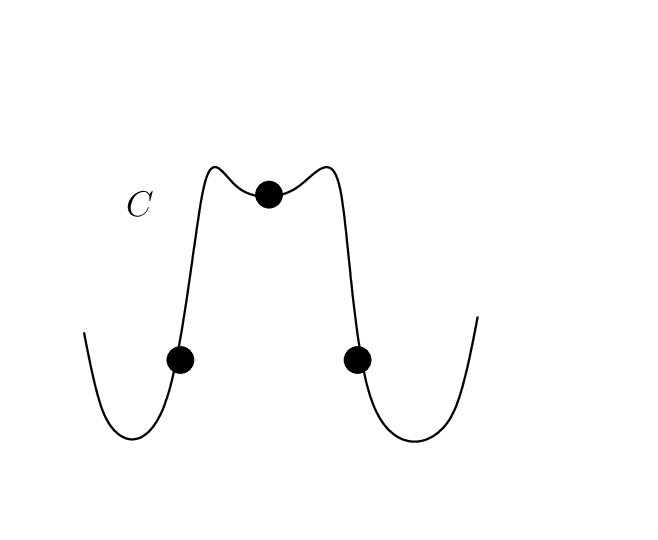} \\
   (\ref{ZL43})
  \end{tabular}
\end{center}

In these three cases, $C$ is not a union of $\F_q$--rational lines. Then, the upper bound is a straightforward consequence of Corollary \ref{CorMajor}.
In case (\ref{ZL41}), if $C_2$ is a union of two $\F_q$--rational lines which do not meet $C_1$ at rational points (it is possible as soon as $q\geq 4$, the details are left to the reader), then $\sharp C (\F_q)=3q+2$ and hence the upper bound is reached.
\end{proof}




Table \ref{CLZ4} gives the parameters of this code for several values of $q$. Comparing the minimum distance with the best known minimum distance for a fixed length and dimension, we see that these codes are almost as good as some best known codes in \cite{grassl} and \cite{minT2}. In addition, we get a $[57, 12, 34]$ code over $\F_7$ which is up to now better than the best known code for these fixed length and dimension.

\begin{table}[!h]
  \centering
  \begin{tabular}{|c|c|c|c|c|}
    \hline
$q$ & $n$ & $k$ & $d$ & Best $d$\\
 &         &    &     &   up to now \\
 \hline 
4 & 21 & 12 & 7 & 7 \\
 \hline 
5 & 31 & 12 & 14 & 14 \\
\hline 
\hline
7 & \textbf{57} & \textbf{12} & \textbf{34} & 33 \\
\hline
\hline
8 & 73 & 12 & 47 & 48 \\
\hline
9 & 91 & 12 & 62 & 62 \\     
    \hline
  \end{tabular}
\medbreak
  \caption{Parameters of the code $C_L (Z, \L_4)$.}
  \label{CLZ4}
\end{table}

\noindent \textbf{Computer construction using \textsc{Magma}.}
A \textsc{Magma} script to construct such a $[57, 12, 34]$ code is available on 
 \url{http://www.lix.polytechnique.fr/Labo/}

\noindent \url{Alain.Couvreur/doc_rech/bestF7.mgm}.


\medbreak

\noindent \textbf{Actualisation of the tables of best codes and generation of other best codes.} The $[57, 12, 34]$ code over $\F_7$ has been sent to \url{www.codetables.de}. The code has been proved by computer to be equivalent to a consta-cyclic code (invariant by shifting by one position and multiplication of the first bit by a fixed constant). Moreover, by computer-aided calculation, the minimum distance has been confirmed to be $34$.
Afterwards, using classical operations on codes (shortening, puncturing, concatenation...) Markus Grassl from \url{www.codetables.de} provided ten new codes beating the best known minimum distances.
These new best codes are available on \url{www.codetables.de}.


\subsubsection{The code $C_L (Z, \L_5)$}
\begin{thm}
  The parameters of $C_L (Z, \L_5)$ are
$$[q^2+q+1, 18, q^2-3q-1].$$
\end{thm}

\begin{proof}
  The length is $n=q^2+q+1$ (as for $C_L (Z, \L_3)$).
The dimension of the linear system of plane quintics is $20$. The vanishing condition at $P$ imposes $3$ independent constraints and hence the dimension of $\Gamma_5$ is $17$.
Thus, the code has dimension $k=18$.
Notice that, in order to have $n \geq k$, the integer $q$ must be above $4$.
The relevant cases appear for $q\geq 5$, which is what is assumed from now on. 
The minimum distance of the code is given by the following result.
\end{proof}

\begin{prop}
Assume that $q\geq 5$.
  Let $C$ be an $\F_q$--rational element of $\Gamma_5$, then
$$
\sharp C (\F_q) \leq 4q+2 
$$ 
and this bound is reached.
\end{prop}

\begin{proof}
  The curve $C$ can be of the form:
  \begin{enumerate}[(i)]
  \item\label{L51} $C=C_1 \cup C_2$, where $C_1$ is an $\F_q$--irreducible conic containing the closed point $P$ and $C_2$ is a cubic (possibly reducible);
  \item\label{L52} $C=C_1 \cup C_2$, where $C_1$ is an $\F_q$--irreducible cubic containing $P$ and $C_2$ is a conic (possibly reducible);
  \item\label{L53} $C=C_1 \cup C_2$, where $C_1$ is an $\F_q$--irreducible quartic containing $P$ and $C_2$ is a line;
  \item\label{L54} $C$ is an $\F_q$--irreducible quintic.
  \end{enumerate}

\noindent The pictures below illustrate these different cases.

\begin{center}
  \begin{tabular}{cc}
    \includegraphics[scale=.5]{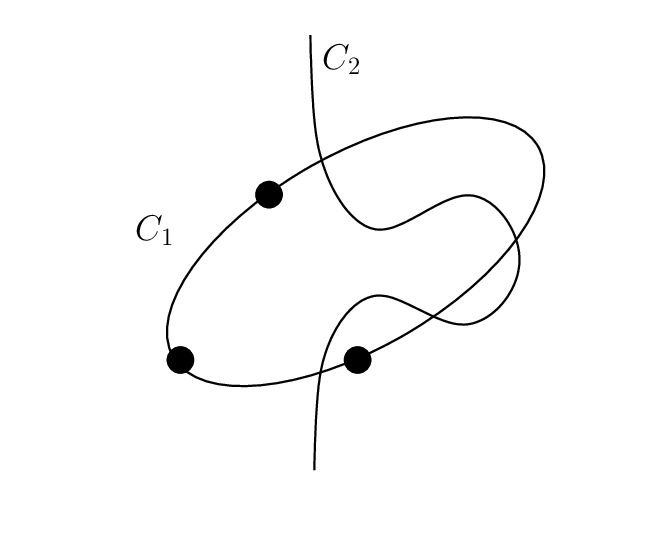} & \includegraphics[scale=.5]{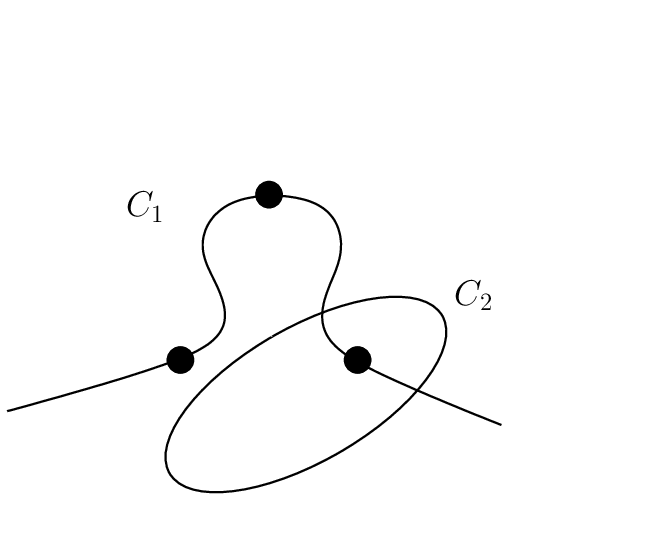} \\
(\ref{L51}) & (\ref{L52}) 
\end{tabular}
\end{center}

\begin{center}
\begin{tabular}{cc}
\includegraphics[scale=.5]{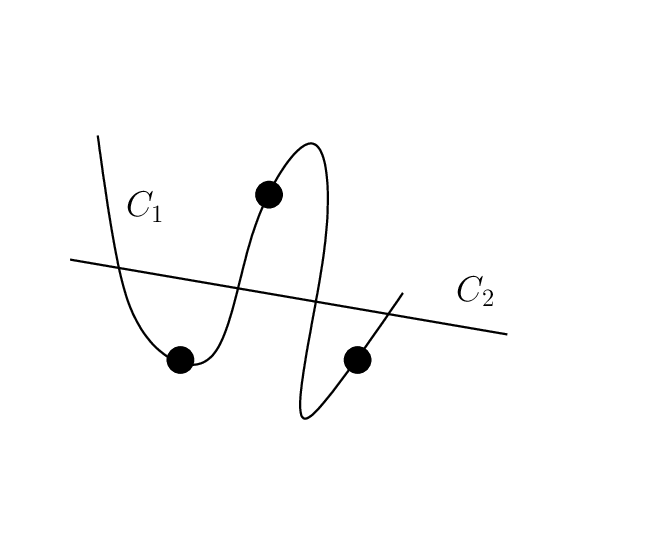} &
\includegraphics[scale=.5]{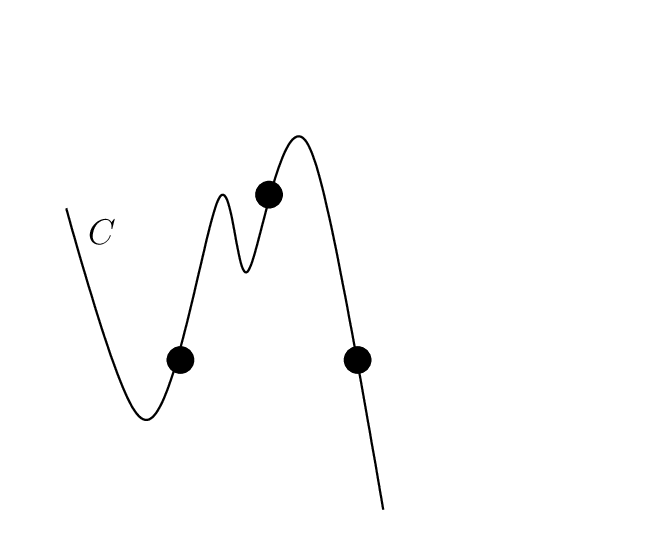} \\
(\ref{L53}) & (\ref{L54})  
  \end{tabular}
\end{center}

Since the curve $C$ cannot be a union of $\F_q$--rational lines, the upper bound is a straightforward consequence of Corollary \ref{CorMajor}.
In case (\ref{L51}), if $C_2$ is a union of three concurrent $\F_q$--rational lines which do not meet $C_1$ at rational points (it is possible as soon as $q\geq 7$), then $\sharp C(\F_q)=4q+2$.
If $q=5$, then the bound is reached in situation (\ref{L53}). Let us give an explicit example. Assume that $P$ is defined by the equations $x^2 + xz + yz$,
 $xy + yz + z^2$ and
$4xz + y^2$. Then the upper bound is reached by the curve of equation
$$
x(x^4 + 2 x^3 y + 3 x^3 z + 3 x^2 y^2 + 4 x^2 y z + 3 x^2 z^2 + 2 x y^3 + \qquad \qquad \qquad
$$
$$
  \qquad \qquad      4 x y^2 z + x y z^2 + 3 x z^3 + 2 y^4 + 4 y^3 z + 2 y^2 z^2 + 4 y z^3 + 
        2 z^4)=0,
$$
which has $22$ rational points.
\end{proof}





Table \ref{CLZ5} gives the parameters of $C_L (Z, \L_5)$ for some values of $q$. It shows that these codes are almost as good as the best known codes. In addition over $\F_9$, we get a $[91, 18, 53]$ code which is better than the best known codes up to now. Indeed, for this length and dimension the best minimum distance given by \cite{grassl} and \cite{minT2} is 52.

\begin{table}[!h]
  \centering
  \begin{tabular}{|c|c|c|c|c|}
    \hline
$q$ & $n$ & $k$ & $d$ & Best $d$\\
    &     &     &     &  up to now \\
 \hline 
5 & 31 & 18 & 9 & 9 \\
\hline
7 & 57 & 18 & 27 & 27 \\
\hline
8 & 73 & 18 & 39 & 40 \\
\hline
\hline
9 & \textbf{91} & \textbf{18} & \textbf{53} & 52 \\     
\hline    
\hline
  \end{tabular}
\medbreak
  \caption{Parameters of the code $C_L (Z, \L_5)$.}
  \label{CLZ5}
\end{table}

\noindent \textbf{Computer construction using \textsc{Magma}.}
A \textsc{Magma} script to construct such a $[91, 18, 53]$ code is available on 
 \url{http://www.lix.polytechnique.fr/Labo/}

\noindent \url{Alain.Couvreur/doc_rech/bestF9.mgm}.

\medbreak

\noindent \textbf{Actualisation of the tables of best known codes.}
The $[91,18,53]$ code over $\F_9$ has been sent to \url{www.codetables.de}.
It has been proved to be equivalent to a cyclic code over $\F_9$ and its dimension has been confirmed to be $53$ by computer-aided calculations.

\section*{Acknowledgements}
The author wishes to thank Christophe Ritzenthaler for inspiring discussions and Markus Grassl from \url{www.codetables.de} for his investigations on the best codes presented in this article.
He is also very grateful to the anonymous referees for their relevant comments and suggestions.

\bibliographystyle{abbrv}
\bibliography{biblio}
\end{document}